\documentclass[12pt]{article}
\usepackage{graphicx}
\usepackage{amsmath,amsthm,amssymb,enumerate}
\usepackage{euscript,mathrsfs}
\usepackage{color}
\usepackage{dsfont}
\usepackage{url}
\usepackage{comment}
\usepackage[left=2cm,right=2cm,top=3.5cm,bottom=3.5cm]{geometry}
\usepackage{color}
\usepackage[framemethod=tikz]{mdframed}
\allowdisplaybreaks

\usepackage{soul}

\catcode`\@=11 \@addtoreset{equation}{section}

\catcode`\@=12

\newtheorem{Theorem}{Theorem}[section]
\newtheorem{Proposition}[Theorem]{Proposition}
\newtheorem{Lemma}[Theorem]{Lemma}
\newtheorem{Corollary}[Theorem]{Corollary}

\theoremstyle{definition}
\newtheorem{Definition}[Theorem]{Definition}

\newtheorem{Remark}[Theorem]{Remark}

\newcommand{\bTheorem}[1]{
	\begin{Theorem} \label{T#1} }
	\newcommand{\eT}{\end{Theorem}}

\newcommand{\bProposition}[1]{
	\begin{Proposition} \label{P#1}}
	\newcommand{\eP}{\end{Proposition}}

\newcommand{\bLemma}[1]{
	\begin{Lemma} \label{L#1} }
	\newcommand{\eL}{\end{Lemma}}

\newcommand{\bCorollary}[1]{
	\begin{Corollary} \label{C#1} }
	\newcommand{\eC}{\end{Corollary}}

\newcommand{\bRemark}[1]{
	\begin{Remark} \label{R#1} }
	\newcommand{\eR}{\end{Remark}}

\newcommand{\bDefinition}[1]{
	\begin{Definition} \label{D#1} }
	\newcommand{\eD}{\end{Definition}}

\newcommand{\Del}{\Delta_x}

\newcommand{\Ds}{\mathbb{D}_x}

\newcommand{\bfomega}{\boldsymbol{\omega}}

\newcommand{\wbfphi}{\widetilde{\bfphi}}

\newcommand{\bfphi}{\boldsymbol{\varphi}}

\newcommand{\bfpsi}{\boldsymbol{\psi}}

\newcommand{\bFormula}[1]{
	\begin{equation} \label{#1}}
	\newcommand{\eF}{\end{equation}}

\newcommand{\Ov}[1]{\overline{#1}}

\newcommand{\Curl}{{\bf curl}_x}
\newcommand{\DC}{C^\infty_c}

\newcommand{\aleq}{\stackrel{<}{\sim}}

\newcommand{\ageq}{\stackrel{>}{\sim}}

\newcommand{\vr}{\varrho}
\newcommand{\vre}{\vr_\ep}

\newcommand{\vue}{\vu_\ep}

\newcommand{\tvu}{{\tilde \vu}}

\newcommand{\vu}{\vc{u}}

\newcommand{\vc}[1]{{\bf #1}}
\newcommand{\Ome}{\Omega_\ep}

\newcommand{\Div}{{\rm div}_x}
\newcommand{\Grad}{\nabla_x}

\newcommand{\dx}{\,{\rm d} {x}}

\newcommand{\dt}{\,{\rm d} t }

\newcommand{\dxdt}{\dx  \dt}
\newcommand{\intO}[1]{\int_{\Omega} #1 \ \dx}

\newcommand{\D}{{\rm d}}

\newcommand{\ep}{\varepsilon}

\newcommand{\br}{ \nonumber \\ }

\def\softd{{\leavevmode\setbox1=\hbox{d}%
		\hbox to 1.05\wd1{d\kern-0.4ex{\char039}\hss}}}
\definecolor{Cgrey}{rgb}{0.85,0.85,0.85}
\definecolor{Cblue}{rgb}{0.50,0.85,0.85}
\definecolor{Cred}{rgb}{1,0,0}
\definecolor{fancy}{rgb}{0.10,0.85,0.10}
\definecolor{amaranth}{rgb}{0.9, 0.17, 0.31}

\newcommand\Cbox[2]{%
	\newbox\contentbox%
	\newbox\bkgdbox%
	\setbox\contentbox\hbox to \hsize{%
		\vtop{
			\kern\columnsep
			\hbox to \hsize{%
				\kern\columnsep%
				\advance\hsize by -2\columnsep%
				\setlength{\textwidth}{\hsize}%
				\vbox{
					\parskip=\baselineskip
					\parindent=0bp
					#2
				}%
				\kern\columnsep%
			}%
			\kern\columnsep%
		}%
	}%
	\setbox\bkgdbox\vbox{
		\color{#1}
		\hrule width  \wd\contentbox %
		height \ht\contentbox %
		depth  \dp\contentbox
		\color{black}
	}%
	\wd\bkgdbox=0bp%
	\vbox{\hbox to \hsize{\box\bkgdbox\box\contentbox}}%
	\vskip\baselineskip%
}

\mdfdefinestyle{MyFrame}{%
	linecolor=black,
	outerlinewidth=1pt,
	roundcorner=5pt,
	innertopmargin=\baselineskip,
	innerbottommargin=\baselineskip,
	innerrightmargin=10pt,
	innerleftmargin=10pt,
	backgroundcolor=white!20!white}

\begin{document}


\title{On the collective effect of a large system of heavy particles immersed in 
a Newtonian fluid}

\author{Marco Bravin$^1$ \and Eduard Feireisl \thanks{The work of E.F. was partially supported by the
		Czech Sciences Foundation (GA\v CR), Grant Agreement
		24-11034S. The Institute of Mathematics of the Academy of Sciences of
		the Czech Republic is supported by RVO:67985840. The research of A.R. has been supported by the Alexander von Humboldt-Stiftung / Foundation. A.Z has been partially supported by the Basque Government through the BERC 2022-2025 program and by the Spanish State Research Agency through BCAM Severo Ochoa CEX2021-001142 and through project PID2020-114189RB-I00 funded by Agencia Estatal de Investigación (PID2020-114189RB-I00 / AEI / 10.13039/501100011033). A.Z. was also partially supported  by a grant of the Ministry of Research, Innovation and Digitization, CNCS - UEFISCDI, project number PN-III-P4-PCE-2021-0921, within PNCDI III.} 
	\and Arnab Roy$^2$ \and Arghir Zarnescu$^{3,4,5}$}

\date{}

\maketitle

\centerline{$^1$ Departmento de Matem\'atica Aplicada y Ciencias de la Computac\'ion}

\centerline{E.T.S.I. Industriales y de Telecomunicac\'ion, Universidad
de Cantabria, 39005 Santander, Spain.}

\centerline{$^*$ Institute of Mathematics of the Academy of Sciences of the Czech Republic}

\centerline{\v Zitn\' a 25, CZ-115 67 Praha 1, Czech Republic}

\centerline{$^2$ Technische Universit\"{a}t Darmstadt}

\centerline{Schlo\ss{}gartenstra{\ss}e 7, 64289 Darmstadt, Germany.}

\centerline{$^3$ BCAM, Basque Center for Applied Mathematics}

\centerline{Mazarredo 14, E48009 Bilbao, Bizkaia, Spain.}

\centerline{$^4$IKERBASQUE, Basque Foundation for Science, }

\centerline{Plaza Euskadi 5, 48009 Bilbao, Bizkaia, Spain.}

\centerline{$^5$`Simion Stoilow" Institute of the Romanian Academy,}

\centerline{21 Calea Grivi\c{t}ei, 010702 Bucharest, Romania.}

\medskip

\begin{abstract}
	
We consider the motion of a large number of heavy particles in 
a Newtonian fluid occupying a bounded spatial domain. When we say ``heavy", we {mean} a particle with a mass density that approaches infinity at an appropriate rate as its radius vanishes. We show that the collective effect of 
heavy particles on the fluid motion is similar to the Brinkman perturbation of the 
Navier--Stokes system identified in the homogenization process. 

\end{abstract}



\tableofcontents

\section{Introduction}
\label{i}

The results of the present paper are complementary to those obtained in our previous work \cite{FeiRoyZar2022}, where the motion of a large number of 
``light'' particles was considered. Here ``light'' means the mass density of each particle is uniformly bounded. In the case of the planar 2d-motion, the result is in a sharp contrast with the homogenization problem, where 
the particles are fixed and produce an additional friction (Brinkman perturbation) 
of the limit Navier--Stokes system. Intuitively, the particles have less influence on the fluid motion if they are allowed to move together with the fluid. Previously, this phenomenon had also been observed in \cite{LacTak} for a single moving ``light'' particle by Lacave, Takahashi.

We consider the motion of a large number of heavy particles in 
a Newtonian fluid. By ``heavy'' we mean that the mass density 
of the particle is large approaching infinity at an appropriate rate when the radius of the particle vanishes. Apparently, the large mass slows down the particle velocity approaching zero in the asymptotic limit. 
We show that the collective effect of the 
heavy particles on the fluid motion is similar to the Brinkman perturbation of the 
Navier--Stokes system identified in the homogenization process. We want to mention that in the case of a single ``heavy'' particle, Iftime and He demonstrate in \cite{He1,He2} that the solution of the fluid-rigid body system approaches a solution of the Navier–Stokes equations (without Brinkman perturbation) in whole space as the diameter of the rigid body approaches zero. The influence of a small rigid body moving in a compressible fluid has recently been analyzed in \cite{FRZ1, FRZ2, BraNec1}.

Let us emphasize that our results concern the genuine fluid--structure 
interaction problem of a system of rigid bodies immersed in  a Newtonian fluid avoiding any simplifying assumptions leading to 
the Stokes--like quasistatic approximation of the fluid equations commonly used in the literature.

Our approach is based on a new result in ``dynamic homogenization'', where 
the particles move with a prescribed velocity, which may be of independent interest. The main idea is conceptually simple, namely the particles will approach a static position in the asymptotic limit. However, its implementation generates complex issues, related fundamentally with the  proper velocity decay in relationship with the size of the particle and their mutual distances. 


First,  
we transform the problem to a fixed spatial domain by means of the change of variables introduced by Inoue and Wakimoto \cite{InWak} and later elaborated by a number of authors, see e.g. 
Cumsille, Takahashi \cite{CumTak}, and \cite{RoyTak}, \cite{BraNec1} for the most recent treatment. The resulting problem is no longer of Navier--Stokes type 
containing a perturbation of the velocity in the time derivative. Next, time regularization is used to obtain a stationary homogenization problem at each fixed time level. At this stage, homogenization limit is performed for any fixed time. It is important that the limit is actually the same for any time depending solely on the spatial distribution of the particles and their radius, see 
Desvilletes, Golse, and Ricci \cite{DesGolRic}, Marchenko and Khruslov  \cite{MarKhr}, or a more recent treatment in \cite{FeiNamNec}. Finally, we relax the time regularization by means of a modification of the Aubin--Lions argument. 

The paper is organized as follows. The problems together with the main results are formulated in Section \ref{P}. Section \ref{k} collects the necessary material concerning homogenization of the stationary Stokes problem. The change in variables is introduced in Section \ref{v}, and the problem is transformed to a fixed spatial domain in Section \ref{N}. The asymptotic limit and proof of the main results are performed in Section \ref{A}. A priori estimates of the velocity of small rigid bodies are established in Section \ref{sec7}. The paper is concluded by a short discussion concerning the extension of the result to the space dimension $d=2$. 

\section{Problem formulation, main result}
\label{P}

We consider a system of $N$ rigid balls $\{ B_{n, \ep} = B(\vc{h}_{n,\ep}, r_\ep) \}_{n=1}^N$, $N = N(\ep)$, centred at points $\vc{h}_{n,\ep} = \vc{h}_{n,\ep}(t)$ and with the common radius $r_\ep$ 
immersed in an incompressible Newtonian fluid occupying a bounded domain $\Omega \subset R^3$. To avoid technicalities, we always assume $\Omega$ is sufficiently smooth, at least of class $C^{2 + \nu}$. The balls move with the rigid velocities  
\begin{equation} \label{P1}
\vc{w}_{n, \ep} (t,x) =  \vc{h}_{n, \ep}'(t) + \bfomega_{n,\ep}(t) \wedge (x - \vc{h}_{n,\ep}(t) ),\ t \in [0,T],\ n = 1,\dots, N(\ep).
\end{equation}

The fluid occupies the domain 
\begin{equation} \label{P2}
	\Omega_{\ep,t} = \Omega \setminus \cup_{n=1}^N B(\vc{h}_{n,\ep}, r_\ep) 
\end{equation}
at each time $t \in [0,T]$. The fluid velocity $\vu_\ep = \vu_{\ep}(t,x)$ 
satisfies the incompressibility constraint 
\begin{equation} \label{P3}
	\Div \vu_\ep (t, \cdot) = 0 \ \mbox{for any}\ t \in [0,T]. 
\end{equation}
The fluid motion is governed by the Navier--Stokes system 
\begin{equation} \label{P4}
\partial_t \vue + \Div (\vue \otimes \vue) + \Grad \Pi_\ep = 
\Del \vu_\ep \ \mbox{in} \ \cup_{t \in (0,T)} \Omega_{\ep,t}.	
	\end{equation} 
For the sake of simplicity, we have set the fluid density as well 
the viscosity coefficient to be one.

We impose the no-slip boundary conditions at the fluid--body interface
\begin{equation} \label{P5}
	\vue|_{\partial \Omega} = 0,\ 
	\vue|_{\partial B_{n,\ep}} = \vc{w}_{n,\ep}|_{\partial B_{n,\ep}}, \ n=1,\dots, N(\ep).
\end{equation}
Accordingly, we consider the initial velocity 
distribution	$\vc{u}_{0,\ep} \in L^2(\Omega; R^3) $ satisfying
\begin{equation}
\label{PMimi}
\Div \vc{u}_{0,\ep} = 0, \quad  \vc{u}_{0,\ep} \cdot \vc{n}|_{\partial \Omega} = 0  \quad \text{ and } \quad \vc{u}_{0,\ep}\cdot  \vc{n}|_{\partial B_{n,\ep}} = \vc{w}_{n,\ep}(0,.) \cdot  \vc{n} |_{\partial B_{n,\ep}} , \ n=1,\dots, N(\ep).
\end{equation}

\subsection{Dynamic homogenization}

 We consider the initial configuration of the particles of \emph{critical} type, where the total (Stokes) capacity of the whole 
ensemble of particles is bounded, while their mutual distance is large enough to keep the collective effect of drag forces balanced with the external forces driving the fluid motion. 

Specifically, we assume
\begin{equation} \label{P1a}
	\frac{r_\ep}{\ep^3} \to 1 \ \mbox{as}\ \ep \to 0,
\end{equation}
while the initial positions of the rigid bodies satisfy
\begin{equation} \label{P6}
|\vc{h}_{i, \ep}(0) - \vc{h}_{j,\ep}(0) | > 2 \ep,\ i\ne j,\ 	
{\rm dist}[\vc{h}_{n,\ep}(0); \partial \Omega] > \ep;\,\, i,j,n\in\{1,\dots,N(\ep)\}
\end{equation}
In addition, we require the translational components of rigid velocities $\vc{w}_{n,\ep}$ to be small, 
\begin{equation} \label{P7}
\sup_{n=1, \dots N(\ep)} \frac{\|\vc{h}'_{n,\ep}\|_{L^\infty(0,T)}}{\ep} \to 0 \ \mbox{as}\ 
\ep \to 0.
\end{equation}
Note carefully that condition \eqref{P7} prevents collisions of the bodies 
in a fixed time interval $[0,T]$ as well as collision of a rigid body with the boundary $\partial \Omega$.

We say that a velocity field $\vue$ is a weak solution of the problem \eqref{P3}, 
\eqref{P4} with initial condition 
$
\vue (0, \cdot) = \vu_{0,\ep} $ satisfying \eqref{PMimi}
if the following holds:
\begin{align} 
	\vue \in L^\infty(0,T; L^2(\Omega; R^3)) &\cap L^2(0,T; W^{1,2}_0(\Omega; R^3)), \label{P8} \\ 
	\vue(t, \cdot)|_{B(\vc{h}_{n,\ep}(t), r_\ep)} &= \vc{w}_{n,\ep}(t, \cdot),\ 
	n=1, \dots, N, \ \mbox{for a.a.}\ t \in (0,T), \label{P9} \\
	\Div \vue &= 0 \ \mbox{in}\ (0,T) \times \Omega, \label{P10} \\
	\int_0^T \intO{ \Big[ \vue \cdot \partial_t \bfphi + 
		[\vue \otimes \vue]: \Grad \bfphi \Big] } \dt &= 
	\int_0^T \intO{ \Grad \vue : \Grad \bfphi } \dt - \intO{ \vu_{0,\ep} \cdot \bfphi (0, \cdot) } 	 
	\label{P11}
	\end{align}
for any test function $\bfphi \in C^1_c ([0,T) \times \Omega; R^3)$, $\Div \bfphi = 0$, satisfying
\begin{equation} \label{P12}
\bfphi (t, \cdot)|_{B(\vc{h}_{n,\ep}(t), r_\ep)} = 0,\ n=1,\dots, N,\ t \in [0,T].
\end{equation} 
	\begin{Remark} \label{bbb}
Unlike the fluid--structure interaction problem considered in Section \ref{FS} below, the rigid velocities are given {\it a priori} in \eqref{P8}--\eqref{P9}. Accordingly, the eligible test functions in the momentum equation \eqref{P11} vanish on the rigid bodies.		
		\end{Remark}

We claim the following result. 

\begin{mdframed}[style=MyFrame]

\begin{Theorem}[{\bf Dynamic homogenization}] \label{TP1}
	
Let $\Omega \subset R^3$ be a bounded domain of class $C^{2 + \nu}$. 	
Suppose the initial distribution of the rigid bodies satisfies 
\begin{align}
&\frac{r_\ep}{\ep^3} \to 1 \ \mbox{as}\ \ep \to 0,\nonumber\\
&\ |\vc{h}_{i, \ep}(0) - \vc{h}_{j,\ep}(0) | > 2 \ep,\ i\ne j,\ 	
{\rm dist}[\vc{h}_{n,\ep}(0); \partial \Omega] > \ep;\,\,i,j,n\in\{1,\dots,N(\ep)\}
\label{P13a}
\end{align}	
In addition, let the rigid body velocities satisfy
\begin{equation} \label{P14}
	\sup_{n=1, \dots N(\ep)} \frac{\|\vc{h}'_{n,\ep}\|_{L^\infty(0,T)}}{\ep^2} \to 0,\ \mbox{as}\ \ep \to 0,\  
	\sup_{n=1, \dots N(\ep)} r_{\ep} \|\bfomega_{n,\ep}\|_{L^\infty(0,T)} \to 0 \ \mbox{as}\ \ep \to 0.
\end{equation}
Let $\vue$ be a weak solution to the moving boundary problem \eqref{P9}--\eqref{P11}, \[ \vue(0, \cdot) = \vu_{0,\ep}\] such that
\begin{equation} \label{P15}
	\| \vue \|_{L^\infty(0,T; L^2(\Omega; R^3))} + \| \vue \|_{L^2(0,T; W^{1,2}_0(\Omega; R^3))} \aleq 1 \ \mbox{uniformly for}\ \ep \to 0. 
\end{equation}

 Then, up to a suitable subsequence

\begin{equation}\label{P13}	
	\ep^3 \sum_{n=1}^{N(\ep)} \delta_{\vc{h}_{n, \ep}(0)} \to \mathfrak{R} \in L^\infty(\Omega) \ \mbox{weakly-(*) in}\ \mathfrak{M}(\Ov{\Omega})
\end{equation}	 and 
\begin{align}
\vue &\to \vu \ \mbox{in}\ L^2((0,T) \times \Omega; R^3) 
\ \mbox{and weakly in} \ L^2(0,T; W_0^{1,2}(\Omega; R^3)), 
\nonumber
\end{align}
where the limit $\vu$ is a weak solution of the Navier--Stokes problem with 
Brinkman friction, 
\begin{align} 
	\Div \vu &=0 , \label{P16}\\
	\int_0^T \intO{ \Big[ \vu \cdot \partial_t \bfphi + [\vu \otimes \vu]: 
		\Grad \bfphi \Big] } \dt &= \int_0^T \intO{ \Big[ \Grad \vu : \Grad \bfphi + 6 \pi \mathfrak{R} \vu \cdot \bfphi \Big]} \dt \br 
	&- \intO{ \vu_0 \cdot \bfphi (0, \cdot)}	
	\label{P17}
	\end{align}
for any $\bfphi \in C^1_c([0,T) \times \Omega; R^3)$, $\Div \bfphi = 0$.	
	
	\end{Theorem}
	
\end{mdframed}	

\begin{Remark} \label{dR1}
Note that hypothesis \eqref{P13a} yields 
\[
\ep^3 N(\ep) \aleq |\Omega| \ \Rightarrow \ep^3\ N(\ep)\to \lambda\ge 0
\]

at least for a suitable subsequence. Accordingly, \eqref{P13} can be written in the form 
\[
\ep^3 \sum_{n=1}^{N(\ep)} \delta_{\vc{h}_{n, \ep}(0)} = \ep^3 N(\ep) \left( \frac{1}{N(\ep)} \sum_{n=1}^{N(\ep)} \delta_{\vc{h}_{n, \ep}(0)} \right), 
\] 
where 
\[
\frac{1}{N(\ep)} \sum_{n=1}^{N(\ep)} \delta_{\vc{h}_{n, \ep}(0)}
\]
is the empirical measure introduced by Desvilletes, Golse, and Ricci \cite{DesGolRic}.  
	
 \end{Remark}


\subsection{Fluid--structure interaction problem}
\label{FS}

In the fluid--structure interaction problem, the velocities of the particles are determined by the mutual interaction with the fluid. Specifically, in addition to 
the no--slip condition \eqref{P5}, continuity of the momenta is imposed. 

For the sake of simplicity, we suppose all particles have the same mass density 
$\vr_{\ep, S} > 0$. The weak formulation of the fluid--structure 
interaction problem reads as follows, see e.g. \cite[Section 2.1]{FeiRoyZar2022}.

\begin{itemize}
\item {\bf Regularity.} 
\begin{align} 
	\vre > 0,\ \vre &\in L^\infty ((0,T) \times R^3) \cap C([0,T]; L^1(\Omega)), \br \vue &\in L^\infty(0,T; \Omega) \cap L^2(0,T; W^{1,2}_0(\Omega)),\ \Div \vue = 0, 
	\br
	\vc{h}_{n,\ep} &\in W^{1,\infty}([0,T]; R^3),\ 
	\bfomega_{n, \ep} \in W^{1,\infty}([0,T]; R^3) ,\ n = 1, \dots, N.
	\label{P20}
	\end{align}	
	
\item {\bf Compatibility.}
\begin{align} 
\vre (t,x) &= \left\{ \begin{array}{l} \vr_{\ep, S} \ \mbox{if}\ 
	 x \in B(\vc{h}_{n,\ep}(t), r_\ep),\ t \in [0,T] \\ 
	 1 \ \mbox{otherwise,} \end{array} \right.	\br
\vue(t, \cdot)|_{B(\vc{h}_{n,\ep}(t), r_\ep)} &= \vc{w}_{n,\ep}(t),\ 
n=1, \dots, N, \ \mbox{for a.a.}\ t \in (0,T). 	 
	\label{P21}
	\end{align}	
	
\item {\bf Mass conservation.}	
\begin{equation} \label{P22}
	\int_0^T \int_{R^3} \Big[ \vre \partial_t \varphi + \vre \vue \cdot \Grad 
	\varphi \Big] \dxdt = - \int_{R^3} \vr_{0,\ep} \varphi (0, \cdot) \dx
\end{equation}
for any $\varphi \in C^1_c ([0,T) \times R^3)$. 

\item {\bf Momentum balance.}	
\begin{equation} \label{P23}
	\int_0^T \intO{ \Big[ \vre \vue \cdot \partial_t \bfphi + 
		\vre[\vue \otimes \vue]: \Grad \bfphi \Big] } \dt = 
	\int_0^T \intO{ \Ds \vue : \Ds \bfphi } \dt - \intO{ \vr_{\ep,0} 
		\vu_{\ep,0} } 
\end{equation}	
for any test function $\bfphi \in C^1_c ([0,T) \times \Omega)$, $\Div \bfphi = 0$,
	\begin{equation} \label{P24}
		\Ds \bfphi (t, \cdot) = 0 
		\ \mbox{on an open neighbourhood of}\ B(\vc{h}_{n,\ep}(t), r_\ep),\  n=1,\dots, N,\ t \in [0,T].
	\end{equation}
The symbol $\Ds$ denotes the symmetric part of the gradient, 
\[
\Ds = \frac{1}{2} \left( \Grad + \Grad^t \right).
\]
\item {\bf Energy dissipation.}
\begin{equation} \label{P25}
\frac{1}{2} \intO{  \vre |\vue|^2 (\tau, \cdot) } + 
\int_0^\tau \intO{ |\Ds \vue |^2 } \dt \leq \frac{1}{2} \intO{ \vr_{\ep,0} |\vu_{\ep, 0} |^2 }  	
	\end{equation}	
for a.a. $\tau \in (0,T)$.	 
	\end{itemize}
	
The following result is a consequence of Theorem \ref{TP1}.

\begin{mdframed}[style=MyFrame]

\begin{Theorem} \label{TP2}
	{\bf (Fluid--structure interaction with heavy particles.)} 	
	Let $\Omega \subset R^3$ be a bounded domain of class $C^{2 + \nu}$. 	
Suppose the initial distribution of the rigid bodies satisfies 
\begin{align*}
&\frac{r_\ep}{\ep^3} \to 1 \ \mbox{as}\ \ep \to 0,\nonumber\\
&\|\vc{h}_{i, \ep}(0) - \vc{h}_{j,\ep}(0) | > 2 \ep,\ i\ne j,\ 	
{\rm dist}[\vc{h}_{n,\ep}(0); \partial \Omega] > \ep; \,\,i,j,n\in\{1,\dots,N(\ep)\}
\end{align*}	 
	Let the rigid densities and the initial velocity of the rigid body satisfy 
	\begin{equation} \label{P27}
	\vr_{\ep, S} \ep^{\frac{19}{2}}  \to \infty \quad \text{ and } \quad \sup_{n \in N(\ep)} \frac{|\vc{h}_{n,\ep}'(0)|}{\ep^2} \to 0 \ 
 \mbox{as}\ \ep \to 0.
	\end{equation}
	Let $(\vre, \vue)$ be a weak solution to the fluid--structure interaction problem \eqref{P20}--\eqref{P25} with bounded initial energy,  
	\begin{equation} \label{P28}
	\intO{ \vr_{\ep,0} |\vu_{\ep, 0}|^2 } \aleq 1 \ \mbox{uniformly for}\ \ep \to 0. 
	\end{equation} 
	Then, up to a suitable subsequence

	\begin{equation} \label{P131}	
\ep^3\sum_{n=1}^{N(\ep)} \delta_{\vc{h}_{n, \ep}(0)} \to \mathfrak{R} \in L^\infty(\Omega) \ \mbox{weakly-(*) in}\ \mathfrak{M}(\Ov{\Omega})
\end{equation}	  and 
	\[
	\vue \to \vu \ \mbox{in}\ L^2((0,T) \times \Omega; R^3) 
	\ \mbox{and weakly in} \ L^2(0,T; W^{1,2}_0(\Omega; R^3)), 
	\]
	where the limit $\vu$ is a weak solution of the Navier--Stokes problem with 
	Brinkman friction, 
	\begin{align} 
		\Div \vu &=0 , \label{P29}\\
		\int_0^T \intO{ \Big[ \vu \cdot \partial_t \bfphi + [\vu \otimes \vu]: 
			\Grad \bfphi \Big] } \dt &=  \int_0^T \intO{ \Big[ \Grad \vu : \Grad \bfphi + 6 \pi \mathfrak{R} \vu \cdot \bfphi \Big]} \dt
		\label{P30} \\
		& - \intO{ \vu_{0}\cdot \bfphi(0,.) } \nonumber 
	\end{align}
	for any $\bfphi \in C^1_c([0,T) \times \Omega; R^3)$, $\Div \bfphi = 0$.	
	
\end{Theorem}

\end{mdframed}
Theorem \ref{TP2} follows directly from Theorem \ref{TP1} combined with the following result providing ontrol of the velocity of a small rigid body based on the ideas of \cite{BraNec2}.

\begin{Proposition}
\label{Impr:contr}
Let $(\vre, \vue)$ be a weak solution to the fluid--structure interaction problem \eqref{P20}--\eqref{P25} with bounded initial energy \eqref{P28}.
If 
\begin{equation}
\label{H1}
\sup_{n \in N(\ep)} \frac{|\vc{h}_{n,\ep}'(0)|}{\ep} \to 0 \quad \text{ and } \quad \sup_{n \in N(\ep)} \frac{r_{\ep}^{1/2}}{\ep m_{n,\ep}} \to 0  \quad \text{ as } \ep \to 0,
\end{equation}	
then, there exists $ \bar{\ep} > 0  $ such that for any $ 0 < \ep < \bar{\ep} $, it holds   
\begin{equation}
\label{bound:velocity}
\| \vc{h}'_{n,\ep} - \vc{h}'_{n,\ep}(0)\|_{L^{\infty}(0,T)} + r_{\ep} \|\bfomega_{n,\ep}-\bfomega_{n,\ep}(0) \|_{L^{\infty}(0,T)} \leq C \frac{r_{\ep}^{1/2}}{m_{n,\ep}},
\end{equation}
where $ C $ does not depends on $ n $ or $ \ep $.

\end{Proposition}

\begin{proof}[Proof of Theorem \ref{TP2}]
The boundedness of the initial energy implies that the velocity field $\vue$ satisfies the 
uniform bound \eqref{P28}. Moreover, assumption \eqref{P27} together with $\frac{r_\ep}{\ep^3} \to 1 \ \mbox{as}\ \ep \to 0$ imply \eqref{H1}. Thus, we have verified the hypothesis of Proposition \ref{Impr:contr}. We deduce from \eqref{bound:velocity} and \eqref{P27} that
\begin{equation}
\label{last:Est}
	\sup_{n=1, \dots N(\ep)} \frac{\|\vc{h}'_{n,\ep}\|_{L^\infty(0,T)}}{\ep^2} \to 0,\ \mbox{as}\ \ep \to 0. 
\end{equation} 
Furthermore, from the bound  $\intO{ \vre |\vue|^2 } \aleq 1$, we have
\begin{equation}\label{bdd:KE}
  \int_{B(\vc{h}_{n,\ep}, r_\ep)} \vr_{\ep, S}| \vc{w}_{n, \ep} |^2 \dx  \aleq 1 
\ \mbox{for any}\ n=1,\dots, N(\ep). 
\end{equation}
Moreover, we have \begin{equation*}
 \int_{B(\vc{h}_{n,\ep}, r_\ep)} \vr_{\ep, S} | \vc{w}_{n, \ep} |^2 =m_{\ep} |\vc{h}'_{n,\ep} |^2 + J_{\ep} |{\bfomega}_{n,\ep}|^2= m_{\ep} |\vc{h}'_{n,\ep} |^2 + \frac{2}{5}m_{\ep} r_{\ep}^2 |{\bfomega}_{n,\ep}|^2 \geq \frac{2}{5}m_{\ep} r_{\ep}^2 |{\bfomega}_{n,\ep}|^2,
\end{equation*}
where $J_{\ep}$ is the moment of inertia of $B(\vc{h}_{n,\ep}(t), r_\ep)$. 

As $\frac{r_\ep}{\ep^3} \to 1 \ \mbox{as}\ \ep \to 0$ and $\vr_{\ep, S}= \frac{m_{\ep}}{r_{\ep}^3}$, we get $m_{\ep}\approx \vr_{\ep, S} \ep^9$. Thus we obtain from \eqref{bdd:KE} that
\begin{equation*}
\vr_{\ep, S} \ep^9 r_{\ep}^2 |{\bfomega}_{n,\ep}|^2 \lesssim 1 \implies r_{\ep}^2 |{\bfomega}_{n,\ep}|^2 \lesssim \frac{1}{\vr_{\ep, S}\ep^9}= \frac{\ep^{1/2}}{\vr_{\ep, S}\ep^{19/2}}.
\end{equation*}
Hence, 
\begin{equation}\label{est:romega}
	\sup_{n=1, \dots N(\ep)} r_{\ep} \|\bfomega_{n,\ep}\|_{L^\infty(0,T)} \to 0 \ \mbox{as}\ \ep \to 0.
\end{equation}
Thus the relations \eqref{est:romega} and \eqref{last:Est} give us \eqref{P14}. Now we can apply Theorem \ref{TP1} in order to deduce Theorem \ref{TP2}.

\end{proof}

\section{Homogenization of the stationary Stokes problem}
\label{k}

 We collect some known results concerning homogenization of the \emph{stationary}
Stokes problem. The distributions of the particles is the same as in \eqref{P13a}. To simplify the notation,  
we write $\vc{h}_{n,\ep}$ instead of $\vc{h}_{n, \ep}(0)$. Similarly, 
we write $\Omega_\ep$ instead of $\Omega_{\ep,0}$.

First, we claim that under the hypotheses of Theorem \ref{TP1} on the initial distribution of the obstacles, there exists a subsequence (not relabelled) such that
\begin{equation} \label{k5a}
\ep^3 \sum_{n=1}^{N(\ep)} \delta_{\vc{h}_{n,\ep}} 
\to \mathfrak{R} \ \mbox{weakly-(*) in}\ \mathfrak{M}(\Ov{\Omega}), 
\end{equation}
where $\mathfrak{R} \in L^\infty(\Omega)$. Indeed, in view of hypothesis \eqref{P6}, we have 
$\ep^3 N(\ep) \aleq |\Omega|$, in particular, the limit \eqref{k5a} exists for a suitable subsequence. Moreover, as any cube of the side $\ep$ can contain at most one point $\vc{h}_{n, \ep}$, we conclude 
\[
\left< \ep^3 \sum_{n=1}^{N(\ep)} \delta_{\vc{h}_{n,\ep}} , \phi \right> \equiv 
\ep^3 \sum_{n=1}^{N(\ep)} \phi (\vc{h}_{n, \ep}) \aleq \int_{\Ov{\Omega}} \phi(x) \ \dx \ \mbox{for any}\ \phi \in C(\Ov{\Omega}), 
\]	 
which yields $\mathfrak{R}$ is bounded absolutely continuous with respect to the Lebesgue measure. This provides the claimed implication \eqref{P13} in Theorem~\ref{TP1}

\subsection{Stokes problem}

Consider the Stokes problem 
\begin{align}
\Del \vc{v}_\ep + \Grad p_\ep &= \vc{f}_\ep  ,\ \Div \vc{v}_\ep = 0 
\ \mbox{in}\ \Omega_\ep \br 
\vc{v}_\ep|_{\partial \Ome} &= 0.\ 
\label{k3}
\end{align}

\begin{Proposition}\label{Pk1}{\bf(Homogenization of Stokes problem.)}
	
	Let $\Omega \subset R^3$ be a bounded domain of class at least $C^{2 + \nu}$.
	Let the particles and their positions satisfy \eqref{P1a}, \eqref{P6}.
	Finally, let
		\begin{equation} \label{k4} 
			\vc{f}_\ep \to \vc{f} \ \mbox{in}\ W^{-1,2}(\Omega; R^3), 
			\end{equation}
			
			\begin{equation} \label{k5}
			\ep^3	\sum_{n=1}^{N(\ep)} \delta_{\vc{h}_{n,\ep}} 
			\to \mathfrak{R} \   \mbox{weakly-(*) in}\ \mathfrak{M}(\Ov{\Omega}) 
			\end{equation}
as $\ep \to 0$, where the limit $\mathfrak{R}$ coincides with an $L^\infty$ function on $\Omega$.
Then the Stokes problem \eqref{k3} admits a unique solution $(\vc{v}_\ep, p_\ep)$,  
and
\begin{equation} \label{k6}		
\vc{v}_\ep \to \vc{v} \ \mbox{weakly in}\ W^{1,2}_0(\Omega; R^3),
\end{equation}
where 
$\vc{v}$ is the unique solution of the Stokes problem 
\begin{align} 
\Del \vc{v} + \Grad p &= \vc{f} + 6\pi \mathfrak{R} \vc{v} \ \mbox{in} \ \Omega, \br 
\vc{v}|_{\partial \Omega} &= 0. \label{k7}
\end{align}
\end{Proposition}

For the proof see Desvillettes, Golse, and Ricci \cite[Theorem 1]{DesGolRic}. A more general setting 
with obstacles of different radii and mutual distances was treated in the monograph \cite[Chapter 4, Theorem 4.7]{MarKhr}. Obstacles of arbitrary shape 
were considered in \cite[Proposition 5.1]{FeiNamNec}.

\subsection{Perturbation of the boundary velocity}
\label{p}

We extend Proposition \ref{Pk1} to functions that do not necessarily vanish on $\partial \Ome$.

\begin{Proposition}[{\bf Boundary velocity}] \label{Pp1}
	
Under the hypotheses of Proposition \ref{Pk1}, suppose that $(\vc{v}_\ep, p_\ep)$ is the unique solution of the Stokes problem
	\begin{align}
		\Del \vc{v}_\ep + \Grad p_\ep &= \vc{f}_\ep  ,\ \Div \vc{v}_\ep = 0 
		\ \mbox{in}\ \Omega_\ep \br 
		\vc{v}_\ep|_{\partial \Ome} &= \vc{w}_\ep|_{\partial \Ome}, 
		\label{p4}
	\end{align}
where 
\[
\Div \vc{w}_\ep = 0,\ \vc{w}_\ep \to 0 \ \mbox{in}\ W^{1,2}_0(\Omega; R^3).
\] 	 
	
	Then 
	\begin{equation} \label{p5}		
		\vc{v}_\ep \to \vc{v} \ \mbox{weakly in}\ W^{1,2}_0(\Omega; R^3),
	\end{equation}
	where 
	$\vc{v}$ is the unique solution of the Stokes problem 
	\begin{align} 
		\Del \vc{v} + \Grad p &= \vc{f} + 6\pi \mathfrak{R} \vc{v} \ \mbox{in} \ \Omega, \br 
		\vc{v}|_{\partial \Omega} &= 0. \label{p6}
	\end{align}
\end{Proposition} 

\begin{proof}
Apply Proposition \ref{Pk1} to the sequence 
$
\vc{v}_\ep - \vc{w}_\ep.
$
	
	\end{proof}

\section{Change of variables}
\label{v}

Our goal is to introduce a suitable change of variables to transform the problem to a fixed domain, specifically, mapping the trajectories $t \mapsto \vc{h}_{n,\ep}(t)$ to straight lines. To this end, we introduce a suitable cut--off function 
\[
\chi \in C^\infty_c [0, 1),\ \chi(z) = \left\{ \begin{array}{l} 1 
\ \mbox{if} \ 0 \leq z < \frac{1}{2}, \\ 
\in [0,1] \ \mbox{if} \ \frac{1}{2} \leq z \leq \frac{3}{4}, \\ 
0 \ \mbox{if}\ z > \frac{3}{4},	 
	\end{array} \right. 
\]
together with the vector field 
\begin{equation} \label{v1}
	\Lambda_\ep (t, x) = \Curl \left[ \sum_{n=1}^{N(\ep)} 
\chi \left( \frac{| x - \vc{h}_{n,\ep}(0) |}{\ep} \right) 	
	\mathbb{T}\Big( x - \vc{h}_{n,\ep}(0) \Big) \cdot \vc{h}'_{n,\ep}(t)   \right],  
\end{equation}	
where
\begin{equation}
\label{def:T}
\mathbb{T}(x) = \left( \begin{array}{ccc}
0  & x_3 & -x_2 \\ -x_3  &0  &x_1 \\  x_2 &-x_1 &0             \end{array} \right),  
\end{equation}	
see Cumsille and Takahashi \cite{CumTak}.

The following properties of $\Lambda_\ep$ can be verified by a direct calculation. 
\begin{itemize}
	\item 
$\Lambda_\ep (t, \cdot) \in C^\infty (R^3; R^3)$ for a.a. $t \in [0,T]$.
\item 
$\Div  	\Lambda_\ep(t, \cdot) = 0$ for a.a. $t \in [0,T]$.
\item  Under the hypotheses of Theorem \ref{TP1}, notably \eqref{P14}, 
\begin{equation} \label{v2}
\frac{ \| D^k_x \Lambda_\ep \|_{L^\infty((0,T) \times R^3; R^{3(k+1)} )}}{\ep^{2 - k}} \to 0 \ \mbox{as}\ \ep \to 0,\ k=0,1,2. 
	\end{equation}
\end{itemize}	

In the definition of $ \Lambda_{\ep}$, we consider rigid velocity that have zero angular velocities. Actually any rigid velocities admit an extension onto $\Omega$ similar to \eqref{v1}. In the following lemma we present the extension. Moreover we show a bound in $W^{1,2}_0(\Omega)$ which we will use to apply Theorem \ref{Pp1}.
\begin{Lemma} \label{LC1}
	Let a family of rigid velocities 
	\[
	\vc{w}_{n,\ep} = \vc{h}'_{n,\ep} + \bfomega_{n,\ep} \wedge (y - \vc{h}_{n,\ep}(0))
	\ \mbox{on}\ B(\vc{h}_{n,\ep}(0), r_\ep ),\ n=1,\dots, N(\ep)
	\]
be given. Then there exists a velocity field $\vc{w}_\ep \in W^{1,2}_0(\Omega; R^3)$ enjoying the following properties: 
		\[
		\Div \vc{w}_\ep  = 0 \ \mbox{in} \ \Omega;
		\]
		
		\[
	\vc{w}_\ep = \vc{w}_{n,\ep} \ \mbox{on}\ B(\vc{h}_{n,\ep} (0), r_\ep),\ 
	n=1,\dots, N(\ep);
		\]
		
		\[
		\| \vc{w}_\ep \|_{W^{1,2}_0(\Omega)} \aleq \left( \max_{n=1, \dots, N(\ep) } | \vc{h}'_{n,\ep} | + 
		r_{\ep} \max_{n=1, \dots, N(\ep) } |\bfomega_n |  \right).
		\]

\end{Lemma}

\begin{proof}

Consider the velocity field
\begin{equation*}
\vc{w}_{\ep}(t,y) = \sum_{n=1}^{N(\ep)} \Curl \left[ \chi \left(2 \frac{|x- \vc{h}_{n,\ep} |}{ r_{\ep}} \right) 	
	\mathbb{T}( x- \vc{h}_{n,\ep}) \cdot  \vc{h}_{n,\ep}'(t) - \chi \left(2 \frac{|x- \vc{h}_{n,\ep} |}{ r_{\ep}} \right) 	\frac{|x- \vc{h}_{n,\ep} |^2}{2} \bfomega_{n, \ep}(t) \right]
\end{equation*}
Then $ \Div\vc{w}_{\ep}   = 0 $, $ \vc{w}_\ep = \vc{w}_{n,\ep} $ on $ B(\vc{h}_{n,\ep} (0), r_\ep)$ for $  n=1,\dots, N(\ep)$ and 
\begin{align*}
\| \nabla \vc{w}_{\ep} \|^2_{L^2(\Omega; R^{3\times 3})}  \lesssim & \, \sum_{n=1}^{N(\ep)} r_{\ep}|\vc{h}_{\ep,n}'(t)|^2 + r_{\ep}^3 |\bfomega_{n, \ep}(t) |^2 \\
\lesssim & \, N(\ep) r_{\ep} \sup_{n=1, \dots, N(\ep)} \left( |\vc{h}_{\ep,n}'(t)|^2 + r_{\ep}^2 |\bfomega_{n, \ep}(t) |^2  \right).
\end{align*}

\end{proof}
\subsection{Flow generated by $\Lambda_\ep$}

In view of \eqref{v2}, the vector field $\Lambda_\ep$ generates a flow $\vc{X}_\ep$. Specifically, $\vc{X}_\ep (t,y)$ is the unique solution of the
non--linear ODE system 
\begin{equation} \label{v3}
	\frac{\D }{\dt}\vc{X}_\ep (t,y) = \Lambda_\ep (t, \vc{X}_\ep(t,y) ),\ 
	\vc{X}_\ep (0, y) = y,\ y \in \Omega.
\end{equation}	

As the velocities $\vc{h}'_{n,\ep}$ are of order ${o}(\ep^2)$, we have 
\begin{equation} \label{v4}
	\Lambda_\ep (t, x) = \vc{h}'_{n,\ep}(t) \ \mbox{if}\ 
x \in B(\vc{h}_{n,\ep}(t), r_\ep ),\ n = 1, \dots, N(\ep).
\end{equation}
Moreover, 
\begin{equation} \label{v5}
	\Lambda_\ep (t, x) = 0 \ \mbox{whenever} \ 
{\rm dist} \left[ x ; \left\{ \vc{h}_{1, \ep}(0), \dots, \vc{h}_{N, \ep}(0) \right\} \right] > \frac{3}{4} \ep.
\end{equation}	

\subsection{Estimates on the flow mapping}
\label{est1}

Differentiating equation \eqref{v3} in $x$ we get 
\begin{equation} \label{v6a}
\frac{\D }{\dt} [D_y \vc{X}_\ep] = \Big[ D_x \Lambda_\ep(t, \vc{X}_\ep (t,y) \Big] \circ [D_y \vc{X}_\ep],\ D_y \vc{X}_\ep(0, \cdot) = \mathbb{I}.
	\end{equation}
In view of the bound \eqref{v2}, we may infer that 
\begin{equation} \label{v6b}
| D_y \vc{X}_\ep (t, y) - \mathbb{I} | \aleq \ep \ \mbox{uniformly for} \ 
T \in [0,T], \ y \in \Omega.
\end{equation}		

Next, differentiating \eqref{v6a}, we obtain 
\begin{align} 
	\frac{\D }{\dt} [D^2_{y,y} \vc{X}_\ep] &= \Big[ D_x \Lambda_\ep(t, \vc{X}_\ep (t,y) \Big] \circ [D^2_{y,y} \vc{X}_\ep] + 
\Big[ D^2_{x,x} \Lambda_\ep(t, \vc{X}_\ep (t,y) \Big] \circ [ D_y \vc{X}_\ep] \circ  [ D_y \vc{X}_\ep],\br D^2_{y,y} \vc{X}_\ep (0, \cdot) &= 0.
\label{v6c}
\end{align}
Consequently, in accordance with \eqref{v2}, 
\begin{equation} \label{v6d}
\Big| D^2_{y,y} \vc{X}_\ep (t,y) \Big| \to 0 \ \mbox{as}\ \ep \to 0 
\ \mbox{uniformly in}\ t \in [0,T], \ y \in \Omega.	
	\end{equation}

\subsection{Inverse of the flow mapping}
\label{est2}

Together with the flow mapping $\vc{X}_\ep = \vc{X}_\ep (t,y)$ determined by \eqref{v3}, we introduce its inverse $\vc{Y}(t,y)$, 
\begin{equation} \label{v6}
\vc{Y}_\ep (t, \vc{X}_\ep (t, y) ) = y,\ y \in \Omega \quad \text{ and } \quad  
\vc{X}_\ep (t, \vc{Y}_\ep (t,x) ) = x,\ x \in \Omega. 	
	\end{equation}
In particular, differentiating \eqref{v6} yields
\begin{equation} \label{v7}
D_x \vc{Y}_\ep (t, \vc{X}_\ep (t, y) ) \circ D_y \vc{X}_\ep (t, y) = \mathbb{I} \quad \text{ and } \quad 
D_y \vc{X}_\ep (t, \vc{Y}_\ep (t,x)) \circ D_x \vc{Y}_\ep (t,x) = \mathbb{I}.
	\end{equation}	
Writing 
\begin{align}
D_x \vc{Y}_\ep (t, \vc{X}_\ep (t, y) ) - \mathbb{I} &= 
D_x \vc{Y}_\ep (t, \vc{X}_\ep (t, y) ) \circ (\mathbb{I} - D_y \vc{X}_\ep (t, y) )\br &= \Big[ D_x \vc{Y}_\ep (t, \vc{X}_\ep (t, y) ) - \mathbb{I} \Big] \circ \Big[ \mathbb{I} - D_y \vc{X}_\ep (t, y) \Big] + \Big[ \mathbb{I} - D_y \vc{X}_\ep (t, y)\Big] .
\nonumber
\end{align}
In view of \eqref{v6b}, we conclude
\begin{equation} \label{v8}
\Big| D_x \vc{Y}_\ep (t, x ) - \mathbb{I} \Big| \aleq \ep \ \mbox{uniformly for}\ 
T \in [0,T],\ x \in \Omega.
\end{equation}
Finally, differentiating once more \eqref{v7} we obtain, by virtue of \eqref{v6d}, 
\begin{equation} \label{v13} 
	\Big| D^2_{x,x} \vc{Y}_\ep (t,x) \Big| \to 0 \ \mbox{as}\ \ep \to 0 
	\ \mbox{uniformly in}\ t \in [0,T], \ x \in \Omega.	
	\end{equation}

\subsection{Estimates of the time derivative}
\label{est3}

Finally, we estimate the time derivative of the flow mapping. To this end, differentiate equation \eqref{v3} to obtain 
\begin{equation} \label{v14}
\Big[ \partial_t  D_y \vc{X}_\ep \Big] (t,y) = D_x \Lambda_\ep \circ
D_y \vc{X}_\ep(t,y) \approx \ep \ \mbox{uniformly for} \ t \in [0,T],\ y \in \Omega, 
\end{equation}
where we have used \eqref{v2}, \eqref{v6b}.

Similarly, differentiating in time \eqref{v6}, we conclude 
\begin{equation} \label{v15}
	\Big| \partial_t D_x \vc{Y}_\ep(t,x) \Big| \aleq \ep \ \mbox{uniformly for}\ 
	t \in [0,T], \ x \in \Omega.
\end{equation}

\subsection{Change of variables}
\label{ch}

We introduce a change of variables 
\[
(t, x) \mapsto (t, y),\ y = \vc{Y}_\ep (t,x),
\]
where $\vc{Y}_\ep$ is the inverse of the flow mas $\vc{X}_\ep$ defined via \eqref{v3}.	As $\Div \Lambda_\ep = 0$, the Jacobian of the transformation is equal to 1. Note that the transformation maps 
the time dependent domains $\Omega_{\ep, t}$ on $\Omega_{\ep,0}$, whereas the moving rigid particles $B(\vc{h}_{n, \ep}(t), r_\ep)$ are mapped onto 
static ones $B(\vc{h}_{n,\ep}(0), r_\ep)$. 

Any function $f = f(t,x)$ is transformed to $\widetilde{f}(t,y)$, 
\begin{equation} \label{w1}
\widetilde{f} (t, \vc{Y}_\ep (t,x)) = f(t,x) \ \mbox{or}\ 
\widetilde{f} (t, y) = f(t, \vc{X}(t,y))  
\end{equation}
\begin{equation} \label{w1a}
\int_\Omega \widetilde{f} (\tau, y) \ \D y = 
\intO{ f(\tau,x) } \ \mbox{for a.a.} \ \tau \in (0,T).
\end{equation}

This type of change of variables ``freezing'' the moving objects is frequently used in the literature devoted to the fluid--structure interaction problems, see e.g. \cite{BraNec1}, Cumsille and Takahashi \cite{CumTak}, Glass et al. 
\cite{GLMS}, Roy and Takahashi \cite{RoyTak} among others.

\section{Navier--Stokes system in the new variables}
\label{N}

Our next goal is to rewrite the Navier--Stokes system in the $(t,y)$ variables introduced in Section \ref{ch}. 

\subsection{Transformed velocity field}

Revisiting the weak formulation of the problem 
\eqref{P9}--\eqref{P11} we introduce the transformed velocity 
\[
\widetilde{\vu}_\ep (t, \vc{Y}_\ep (t,x)) = \vue (t, x). 
\]
Accordingly, 
\[
\nabla_y \widetilde{\vu}_\ep \Big( t, \vc{Y}_\ep (t,x) \Big) \cdot [D_x \vc{Y}_\ep] \Big( t, \vc{X}_\ep (t, 
\vc{Y}(t,x) \Big) 
= \Grad \vue (t, x).
\]
As the original velocity coincides with the rigid velocities on the particles (cf. \eqref{P9}), we get 
\begin{equation} \label{N1}
\widetilde{\vu}_\ep (t, y) = \vc{w}_{n,\ep} (t,y) \ \mbox{for}\ y \in B(\vc{h}_{n,\ep}(0), r_\ep),\ n=1,\dots, N(\ep).
\end{equation}

\subsection{Weak formulation in the new variables}

Applying the change of variables to a test function $\bfphi$, we get 
\[
\widetilde{\bfphi} (y, \vc{Y}_\ep (t,x)) = \bfphi(t,x).
\]
Consequently, 
\[
\nabla_y \wbfphi\Big( y, \vc{Y}_\ep(t,x) \Big) \cdot D_x \vc{Y}_\ep \Big(t, \vc{X}_\ep(t, \vc{Y}_\ep(t,x)) \Big)  = \Grad \bfphi(t,x), 
\]
and 
\[
\partial_t \wbfphi \Big(t, \vc{Y}_\ep(t, x) \Big) + [\nabla_y \wbfphi] (t, \vc{Y}_\ep(t, x) ) \cdot [\partial_t \vc{Y}_\ep]\Big( t, \vc{X}_\ep (t, \vc{Y}_\ep (t,x) ) \Big) = 
\partial_t \bfphi(t,x).
\]

The transformed test functions are smooth in the spatial variable $x$ and continuously differentiable in time. Moreover, in accordance with \eqref{P12}, 
\begin{equation} \label{N3}
\wbfphi \in C^1_c ([0,T) \times \Omega_{\ep,0}).	
	\end{equation}

Rewriting the weak formulation \eqref{P11} in the new variables $(t,y)$ we obtain
\begin{align}
\int_0^T &\int_{\Omega} \Big[ \widetilde{\vu}_\ep(t,y) \cdot \partial_t \wbfphi (t,y) + \widetilde{\vu}_\ep (t,y) \cdot \nabla_y \widetilde{\bfphi} (t,y) \cdot 
[\partial_t \vc{Y}_\ep] (t, \vc{X}_\ep (t,y) ) \Big] \D y \br 
&+ \int_0^T \int_{\Omega} [ \widetilde{\vu}_\ep (t,y) \otimes \widetilde{\vu}_\ep (t,y) ] : 
\Big[ \nabla_y \widetilde{\bfphi} (t,y) \circ [D_x \vc{Y}_\ep ](t, \vc{X}_\ep(t,y)) \Big] \D y \dt \br 
&= \int_0^T \int_{\Omega} \Big[ 
\nabla_y \widetilde{\vu}_\ep (t,y) \cdot [D_x \vc{Y}_\ep] (t, \vc{X}_\ep (t,y)) \Big]: 
\Big[ 
\nabla_y \widetilde{\bfphi}(t,y) \cdot [D_x \vc{Y}_\ep ] (t, \vc{X}_\ep (t,y)) \Big] \D y \dt
\label{w4}
	\end{align}
provided $\bfphi(0, \cdot) = \wbfphi (0, \cdot) = 0$.	

The apparent drawback of \eqref{w4} is that neither $\tvu_\ep$ nor $\wbfphi$ are solenoidal. To overcome this problem,  
write 
\begin{align}
\nabla_y \widetilde{\vu}_\ep(t,y) \cdot [D_x \vc{Y}_\ep](t, \vc{X}_\ep(t,y)) &= 
\nabla_y \Big[ \widetilde{\vu}_\ep(t,y) \cdot [D_x \vc{Y}_\ep ](t, \vc{X}_\ep(t,y)) \Big] \br &- 
\widetilde{\vu}_\ep \cdot [D^2_{x,x} \vc{Y}_\ep] (t, \vc{X}_\ep(t,y)) \cdot \nabla_y \vc{X}_\ep(t,y),
\nonumber
\end{align}
and, similarly,
\begin{align}
\nabla_y \widetilde{\bfphi}(t,y) \cdot [D_x \vc{Y}_\ep] (t, \vc{X}_\ep (t,y)) &= 
\nabla_y \Big[ \widetilde{\bfphi}(t,y) \cdot [D_x \vc{Y}_\ep] (t, \vc{X}_\ep (t,y)) \Big] \br  &- 
\widetilde{\bfphi} \cdot [D^2_{x,x} \vc{Y}_\ep (t, \vc{X}(t,y)) ] \cdot \nabla_y \vc{X}_\ep(t,y).
\nonumber
\end{align}

Now, we introduce a new velocity field 
\begin{equation} \label{w5}
	\vc{v}_\ep (t,y) = \widetilde{\vu}_\ep(t,y) \cdot [ D_x \vc{Y}_\ep ] (t, \vc{X}_\ep (t,y)).  
	\end{equation}
Similarly, we consider a new class of test functions,  
\begin{equation} \label{w6}
\bfpsi(t,y) = 	\widetilde{\bfphi}(t,y) \cdot [D_x \vc{Y}_\ep] (t, \vc{X}_\ep (t,y)).
	\end{equation}
It is easy  to check that 
\begin{equation} \label{N7}
{\rm div}_y \vc{v}_\ep = 0,\ {\rm div}_\ep \bfpsi = 0 .
\end{equation}
Moreover, in accordance with \eqref{N1}, 
\begin{equation} \label{N8}
\vc{v}_\ep|_{\partial \Omega_{\ep,0}} = \vc{w}_\ep, 
\end{equation}	
where $\vc{w}_\ep$ is the extension of the rigid velocities constructed 	
in Lemma \ref{LC1}. Finally, by virtue of \eqref{N3}, 
\begin{equation} \label{N9}
	\bfpsi \in C^1_c ([0,T) \times \Omega_{\ep, 0}). 
\end{equation}

The ultimate goal of this part is rewriting the weak formulation \eqref{w4}
in terms of the new variables $\vc{v}_\ep$ and $\bfpsi$. This is a bit tedious 
but rather straightforward process, where we systematically use the identity 
\begin{equation} \label{w11}
	D_y \vc{X}_\ep (t,y) \circ D_x \vc{Y}_\ep  (t, \vc{X}_\ep(t,y)) = \mathbb{I}.
\end{equation}	

\medskip 

{\bf Step 1}

\medskip 	
The integral on the right--hand side of \eqref{w4} reads
\begin{align}
	\int_0^T &\int_{\Omega} \Big[ 
	\nabla_y \widetilde{\vu}_\ep (t,y) \cdot [D_x \vc{Y}_\ep ](t, \vc{X}_\ep(t,y)) \Big]: 
	\Big[ 
	\nabla_y \widetilde{\bfphi}(t,y) \cdot [ D_x \vc{Y}_\ep ](t, \vc{X}_\ep (t,y)) \Big] \D y \dt
\br &=	\int_0^T \int_{\Omega} \Big( \nabla_y \vc{v}_\ep (t,y) - \widetilde{\vu}_\ep \cdot [D^2_{x,x} \vc{Y}_\ep] (t, \vc{X}_\ep (t,y)) \cdot [D_y \vc{X}_\ep ](t,y) \Big) \br 
&\quad \quad \quad \quad : 
\Big( \nabla_y \bfpsi(t,y) - \widetilde{\bfphi} \cdot [D^2_{x,x} \vc{Y}_\ep (t, \vc{X}_\ep (t,y))] \cdot [D_y \vc{X}_ep ] (t,y) \Big) \ \D y \dt \br 
&= \int_0^T \int_{\Omega} \Grad \vc{v}(t,y) : \Grad \bfpsi (t,y) \ \D y \dt \br
&\quad - \int_0^T \int_{\Omega}  \nabla_y \vc{v}(t,y) : [D^2_{x,x} \vc{Y}_\ep] (t, \vc{X}_\ep (t,y)) \cdot [D_y \vc{X}_\ep] (t,y) \cdot [D_y \vc{X}_\ep] (t,y) \cdot \bfpsi \ \D y \dt \br 
&\quad - \int_0^T \int_{\Omega} \widetilde{\vu}_\ep \cdot [D^2_{x,x} \vc{Y}_\ep] (t, \vc{X}(t,y)) \cdot [D_y \vc{X}_\ep ](t,y) : \nabla_y \bfpsi \ \D y \dt \br 
&\quad  + \int_0^T \int_{\Omega} \widetilde{\vu}_\ep \cdot [D^2_{x,x} \vc{Y}_\ep] (t, \vc{X}_\ep (t,y)) \cdot [D_y \vc{X}_\ep] (t,y) \br \quad &\quad \quad \quad \quad \quad \cdot [D^2_{x,x} \vc{Y}_\ep ](t, \vc{X}_\ep (t,y)) \cdot [D_y \vc{X}_\ep ](t,y) \cdot [D_y \vc{X}_\ep ](t,y) \cdot \bfpsi \ \D y \dt.
	\label{w7}
	\end{align}
	
	\medskip 
	
	{\bf Step 2}
	
	\medskip 	
	
Similarly, 
\begin{align}  
\int_0^T &\int_{\Omega} \widetilde{\vu}_\ep (t,y) \otimes \widetilde{\vu}_\ep(t,y) : 
\Big[ \nabla_y \widetilde{\bfphi} (t,y) \cdot [D_x \vc{Y}_\ep] (t, \vc{X}_\ep (t,y)) \Big] \D y \dt
\br &= \int_0^T \int_{\Omega} [ \widetilde{\vu}_\ep (t,y) \otimes \widetilde{\vu}_\ep ] (t,y) : 
\nabla_y \bfpsi(t,y) \D y \dt
\br &-\int_0^T \int_{\Omega} \widetilde{\vu}_\ep (t,y) \otimes \widetilde{\vu}_\ep (t,y) : 
\Big[ D_y \vc{X}_\ep (t,y) \cdot[ D^2_{x,x} \vc{Y}_\ep ](t, \vc{X}_\ep (t,y)) 
\cdot [D_y \vc{X}_\ep (t,y)] \Big]\cdot \bfpsi(t,y) \D y \dt.
\label{w8}
\end{align}

\medskip 

{\bf Step 3}

\medskip 	

Finally, the first integral in \eqref{w4} can be handled as follows:
\begin{align} 
\int_0^T &\int_{\Omega} \Big[ \widetilde{\vu}_\ep (t,y) \cdot \partial_t \widetilde{\bfphi} (t,y) + \widetilde{\vu}_\ep (t,y) \cdot \widetilde{\bfphi} (t,y) \cdot 
[\partial_t \vc{Y}_\ep] (t, \vc{X}_\ep (t,y) ) \Big] \D y \br 
&= \int_0^T \int_{\Omega}  \widetilde{\vu}_\ep (t,y) \cdot \partial_t \Big[
[D_y \vc{X}_\ep] (t,y) \cdot  [D_x \vc{Y}_\ep ](t, \vc{X}_\ep (t,y)) \widetilde{\bfphi} (t,y) \Big] \D y \dt\br &\quad +   \int_0^T \int_{\Omega}  \widetilde{\vu}_\ep (t,y) \cdot \nabla_y \widetilde{\bfphi} (t,y) \cdot 
[\partial_t \vc{Y}_\ep] (t, \vc{X}_\ep (t,y) ) \D y \br 
&= \int_0^T \int_{\Omega} \Big[ [D_y \vc{X}_\ep] (t,y) \cdot  \vc{v}_\ep (t,y) \Big] \cdot \partial_t \Big[
[D_y \vc{X}_\ep] (t,y) \cdot  \bfpsi (t,y) \Big] \D y
\br &\quad +   \int_0^T \int_{\Omega}  \widetilde{\vu}_\ep (t,y) \cdot \nabla_y \Big[ D_y \vc{X}_\ep  (t,y) \cdot \bfpsi (t,y) \Big] \cdot 
[\partial_t \vc{Y}_\ep ](t, \vc{X}_\ep (t,y) ) \D y \br 
&= \int_0^T \int_{\Omega} \Big[ [D_y \vc{X}_\ep] (t,y) \cdot  \vc{v}_\ep(t,y)  \cdot 
[D_y \vc{X}_\ep ] (t,y) \Big] \cdot  \partial_t \bfpsi (t,y) \D y
\br
&\quad  +\int_0^T \int_{\Omega} \Big[ D_y \vc{X}_\ep (t,y) \cdot  \vc{v}_\ep (t,y) \Big] \cdot \partial_t 
\Big( \nabla_y \vc{X}_\ep (t,y) \Big) \cdot  \bfpsi (t,y) \Big] \D y
\br &\quad +   \int_0^T \int_{\Omega}  \widetilde{\vu}_\ep (t,y) \cdot \nabla_y \Big[ D_y \vc{X}_\ep  (t,y) \cdot \bfpsi (t,y) \Big] \cdot 
[\partial_t \vc{Y}_\ep] (t, \vc{X}_\ep (t,y) ) \D y . 
\label{w9}
\end{align}

Summing up the previous identities, we obtain the weak formulation in the new coordinates:
\begin{align}
	\int_0^T &\int_{\Omega} \Big[ [D_y \vc{X}_\ep ] (t,y) \cdot  \vc{v}_\ep (t,y)  \cdot 
	[D_y \vc{X}_\ep] (t,y) \Big] \cdot  \partial_t \bfpsi (t,y) \D y
	\br
	&+\int_0^T \int_{\Omega} [\widetilde{\vu}_\ep (t,y) \otimes \widetilde{\vu}_\ep (t,y)]  : 
	\nabla_y \bfpsi(t,y) \D y \dt
	= \int_0^T \int_{\Omega} \Grad \vc{v}_\ep (t,y) : \Grad \bfpsi (t,y) \ \D y \dt + \mathcal{R}_\ep
\label{N14}
\end{align}
for any $\bfpsi \in C^1_c((0,T) \times \Omega_{\ep,0}$, ${\rm div} \bfpsi = 0$.
The remainder term $\mathcal{R}_\ep$ reads	
\begin{align}	
	\mathcal{R}_\ep = 
	&- \int_0^T \int_{\Omega}  \nabla_y \vc{v}_\ep(t,y)  \cdot [D^2_{x,x} \vc{Y}_\ep] (t, \vc{X}_\ep (t,y))\cdot [D_y \vc{X}_\ep] (t,y) \cdot [D_y \vc{X}_\ep ](t,y) \cdot \bfpsi \ \D y \dt \br 
	&- \int_0^T \int_{\Omega} \widetilde{\vu}_\ep(t,y) \cdot [D^2_{x,x} \vc{Y}_\ep] (t, \vc{X}_\ep (t,y)) \cdot [D_y \vc{X}_\ep (t,y)] : \nabla_y \bfpsi \ \D y \dt \br 
	&+ \int_0^T \int_{\Omega} \widetilde{\vu}_\ep \cdot [D^2_{x,x} \vc{Y}_\ep ](t, \vc{X}(t,y)) \cdot [D_y \vc{X}_\ep](t,y) \br &\quad \quad \quad \quad \cdot [D^2_{x,x} \vc{Y}_\ep] (t, \vc{X}_\ep(t,y)) \cdot [D_y \vc{X}_\ep] (t,y) \cdot [D_y \vc{X}_\ep](t,y) \cdot \bfpsi \ \D y \dt \br 
	&- \int_0^T \int_{\Omega} \Big[ [D_y \vc{X}_\ep ] (t,y) \cdot  \vc{v}_\ep (t,y) \Big] \cdot \partial_t 
	\Big( D_y \vc{X}_\ep (t,y) \Big) \cdot  \bfpsi (t,y) \Big] \D y
	\br &-   \int_0^T \int_{\Omega}  \widetilde{\vu}_\ep (t,y) \cdot \nabla_y \Big[ [D_y \vc{X}_\ep]  (t,y) \cdot \bfpsi (t,y) \Big] \cdot 
	[\partial_t \vc{Y}_\ep ](t, \vc{X}_\ep (t,y) ) \D y \br 
	&+\int_0^T \int_{\Omega} \widetilde{\vu}_\ep (t,y) \otimes \widetilde{\vu}_\ep (t,y) : 
	\Big[ [D_y \vc{X}_\ep] (t,y) \cdot [D^2_{x,x} \vc{Y}_\ep ](t, \vc{X}_\ep(t,y)) \cdot
	[D_y \vc{X}_\ep ](t,y) \Big]\cdot \bfpsi(t,y) \D y \dt. 
	\label{w10}
	\end{align}
	
\section{Asymptotic limit}
\label{A}	

We complete the proof of Theorem \ref{TP1} by performing the limit $\ep \to 0$ in the transformed system \eqref{w10}. As the position of the particles are fixed, we shall write 
\[
\Omega_\ep \ \mbox{instead of}\ \Omega_{\ep, 0}, \ \mbox{and}\  
\vc{h}_{n,\ep} \ \mbox{instead of}\ \vc{h}_{n, \ep}(0).
\]

\subsection{Remainder term}

Our first goal is to show that the remainder term $\mathcal{R}_\ep$ is small in the dual $W^{-1,2}(\Omega)$ norm. To this end, first observe that both 
$\widetilde{\vu}_\ep$, $\vc{v}_\ep$ inherit the uniform bounds from hypothesis 
\ref{P15}, specifically, 
\begin{align} 
	\| \widetilde{\vu}_\ep \|_{L^\infty(0,T; L^2(\Omega; R^3))} &+ \| \widetilde{\vu}_\ep \|_{L^2(0,T; W^{1,2}_0(\Omega; R^3))} \aleq 1, \br
	\| \vc{v}_\ep \|_{L^\infty(0,T; L^2(\Omega; R^3))} &+ \| \vc{v}_\ep \|_{L^2(0,T; W^{1,2}_0(\Omega; R^3))} \aleq 1\ 
	 \mbox{uniformly for}\ \ep \to 0. \label{A1}
\end{align}

Combining the uniform bounds established on the flow mapping established in 
Sections \ref{est1}--\ref{est3}, notably \eqref{v6b}, \eqref{v6d}, \eqref{v14}
for $\vc{X_\ep}$, and \eqref{v8}, \eqref{v13}, \eqref{v15} for $\vc{Y}_\ep$,  
with \eqref{A1}, we easily deduce: 
\begin{align} 
\nabla_y \vc{v}_\ep : [D^2_{x,x} \vc{Y}_\ep] (t, \vc{X}_\ep (t,y))\cdot [D_y \vc{X}_\ep]  \cdot [D_y \vc{X}_\ep ] \to 0 \ \mbox{in}\ 
L^\infty(0,T; L^2(\Omega; R^3)), \br
\widetilde{\vu}_\ep \cdot [D^2_{x,x} \vc{Y}_\ep] (t, \vc{X}_\ep (t,y)) \cdot [D_y \vc{X}_\ep] \to 0 \ \mbox{in}\ L^2(0,T; L^6(\Omega; R^3)), \br
\widetilde{\vu}_\ep \cdot [D^2_{x,x} \vc{Y}_\ep ](t, \vc{X}(t,y)) \cdot [D_y \vc{X}_\ep] \cdot [D^2_{x,x} \vc{Y}_\ep] (t, \vc{X}_\ep(t,y)) \cdot [D_y \vc{X}_\ep]  \cdot [D_y \vc{X}_\ep] \to 0 \ \mbox{in}\ L^2(0,T; L^6(\Omega; R^3)), \br \Big[ [D_y \vc{X}_\ep ]  \cdot  \vc{v}_\ep  \Big] \cdot \partial_t 
\Big( D_y \vc{X}_\ep  \Big) \to 0 \ \mbox{in}\ L^2(0,T; L^6(\Omega; R^3)),\br
\widetilde{\vu}_\ep  \cdot [D^2_{y,y} \vc{X}_\ep] \cdot 
[\partial_t \vc{Y}_\ep ](t, \vc{X}_\ep (t,y) ) \to 0 \ \mbox{in}\ L^2(0,T; L^6(\Omega; R^3)), \br
[\widetilde{\vu}_\ep  \otimes \widetilde{\vu}_\ep] : 
\Big[ [D_y \vc{X}_\ep]  \cdot [D^2_{x,x} \vc{Y}_\ep ](t, \vc{X}_\ep(t,y)) \cdot
[D_y \vc{X}_\ep ] \Big] \to 0 \ \mbox{in}\ L^\infty(0,T; L^1(\Omega; R^3)) \cap 
L^2(0,T; L^3(\Omega; R^3))
	\label{A2}	  
\end{align}	
as $\ep \to 0$. Here, we have used the Sobolev embedding $W^{1,2} \hookrightarrow 
L^6$. 

\subsection{Strong convergence of the velocity field}
\label{scv}

Using the uniform bounds \eqref{P15}, \eqref{A1} to obtain, up to a suitable subsequence, 
\begin{align} \label{A3}
	\vue \to \vu \ \mbox{weakly-* in}\ L^\infty(0,T; L^2(\Omega; R^3)), \ \mbox{and}
	\ \mbox{weakly in}\ L^2(0,T; W^{1,2}_0(\Omega; R^3)), \\
		\widetilde{\vu}_\ep \to \widetilde{\vu} \ \mbox{weakly-* in}\ L^\infty(0,T; L^2(\Omega; R^3)), \ \mbox{and}
	\ \mbox{weakly in}\ L^2(0,T; W^{1,2}_0(\Omega; R^3)),
	\label{A4}\\ 
	\label{A5}
		\vc{v}_\ep \to \vc{v} \ \mbox{weakly-* in}\ L^\infty(0,T; L^2(\Omega; R^3)), \ \mbox{and}
	\ \mbox{weakly in}\ L^2(0,T; W^{1,2}_0(\Omega; R^3)).
	\end{align}
Moreover, it is easy to see that the three limits coincide, $\vu = \widetilde{\vu} = \vc{v}$. Indeed, in view of \eqref{w5}, 	
\[
\vc{v}_\ep (t,y) =\widetilde{\vu}_\ep(t,y) + \widetilde{\vu}_\ep(t,y) \cdot \Big( [ D_x \vc{Y}_\ep ] (t, \vc{X}_\ep (t,y)) - \mathbb{I} \Big);
\]
whence, by virtue of \eqref{v8}, $\tvu = \vc{v}$. To show $\vu = \tvu$ first observe 
\[
\int_0^T \int_{\Omega} \widetilde{\vu}_\ep (t,y) \cdot \wbfphi (t,y) \ \D y \dt = 
\int_0^T \intO{ \vue (t,x) \cdot \bfphi (t,x) } \dt, 
\] 
where 
\[
\wbfphi (t,y) = \bfphi (t, \vc{X}_\ep (t, y)).
\]
Consequently, fixing the test function $\bfphi$, we get 
\begin{align}
\int_0^T & \int_{\Omega} \widetilde{\vu} (t,y) \cdot \bfphi (t,y) \ \D y \dt = 
\lim_{\ep \to 0} \int_0^T \int_{\Omega} \widetilde{\vu}_\ep (t,y) \cdot \bfphi (t,y) \ \D y \dt \br
&= \lim_{\ep \to 0} \int_0^T \int_{\Omega} \widetilde{\vu}_\ep (t,y) \cdot \Big( \bfphi (t,y) - \bfphi(t, \vc{X_\ep}(t,y)  \Big) \D y \dt 
+ \lim_{\ep \to 0} \int_0^T \intO{ \vue (t,x) \cdot \bfphi (t,x) } \dt \br 
&= \int_0^T \intO{ \vu (t,x) \cdot \bfphi (t,x) } \dt, 
\nonumber
\end{align}
where we have used bounded of $\widetilde{\vu}_\ep$ in $L^\infty(0,T; L^2(\Omega; R^3)$, together with 
\[
\sup_{t \in [0,T], y \in \Omega} \Big|\bfphi (t,y) - \bfphi(t, \vc{X_\ep}(t,y) \Big| \to 0
\ \mbox{as}\ \ep \to 0
\]
for any $\bfphi \in C^1_c((0,T) \times \Omega; R^3)$. As $\bfphi$ was arbitrary, 
we conclude $\tvu = \vu$.

Our next goal is to establish the strong convergence, specifically, 
\begin{equation} \label{A6}
	\vc{v}_\ep \to \vc{v} \ \mbox{in}\ L^2((0,T) \times \Omega; R^3).
\end{equation} 
Repeating the above arguments, we can show easily that \eqref{A6} yields
the same for the original velocity,  
\begin{equation} \label{A7}
		\vc{u}_\ep \to \vc{u} \ \mbox{in}\ L^2((0,T) \times \Omega; R^3).
	\end{equation} 
	
To show \eqref{A7} we use a modification of the Aubin--Lions argument. To this end, we need a \emph{restriction} operator $R_\ep$ enjoying the following properties:
\begin{align} 
R_\ep : W^{1,2}_0 (\Omega; R^3) &\to W^{1,2}_0 (\Omega_\ep; R^3), \br
\| R_\ep [\bfphi ] \|_{W^{1,2} (\Omega)} &\aleq \| \bfphi \|_{W^{1,2}(\Omega; R^3)}, \br 
{\rm div}_y \bfphi = 0 \ &\Rightarrow \ {\rm div}_y R_\ep [\bfphi] = 0, \br
R_\ep [\bfphi] &\to \bfphi \ \mbox{in}\ L^2(\Omega; R^3) \ \mbox{as}\ \ep \to 0 \ \mbox{for any}\ 
\bfphi \in W^{1,2}_0(\Omega).  
\label{A8}
\end{align}
There are several constructions of a restriction operator available, see e.g. 
Allaire \cite{Allai4}, Diening et. al \cite{DieRuzSch}, Tartar \cite{Tar1}. Note that these constructions 
	of the restriction operator adapt to the non--periodic distribution of obstacles in a straightforward manner.

Now, consider a function 
\[
\bfpsi = \chi R_\ep [\phi],\ \chi = \chi(t), \chi \in C^\infty_c(0,T),\ 
\phi \in \DC(\Omega; R^3),\ {\rm div}_y \phi = 0.  
\] 
By virtue of the properties of the restriction operator, the function $\bfpsi$ is an eligible test function in the weak formulation \eqref{N14}. 
In view of the uniform bounds \eqref{A1}, \eqref{A2}, we deduce that the function 
\[
t \mapsto \int_{\Omega} \Big[ [D_y \vc{X}_\ep ] \cdot \vc{v}_\ep \cdot [D_y \vc{X}_\ep ] R_{\ep}[\phi] \Big] (t,y) \D y  ,\ t \in [0,T] 
\]
has a bounded time derivative in $L^q(0,T)$ for some $q > 1$ uniformly for $\ep \to 0$.

Next, 
\begin{align}
t \mapsto \int_{\Omega} \Big[ [D_y \vc{X}_\ep ] \cdot \vc{v}_\ep \cdot [D_y \vc{X}_\ep ] R_{\ep}[\phi] \Big] (t,y) \D y &=  \int_{\Omega} \Big[ [D_y \vc{X}_\ep ] \cdot \vc{v}_\ep \cdot [D_y \vc{X}_\ep ] \phi \Big] (t,y) \D y \br 
&+  \int_{\Omega} \Big[ [D_y \vc{X}_\ep ] \cdot \vc{v}_\ep \cdot [D_y \vc{X}_\ep ] ( R_\ep[\phi] - \phi ) \Big] (t,y) \D y,
\nonumber
\end{align}
where, by virtue of \eqref{A8} and the bounds \eqref{A1}, 
\[
\int_{\Omega} \Big[ [D_y \vc{X}_\ep ] \cdot \vc{v}_\ep \cdot [D_y \vc{X}_\ep ] ( R_\ep[\phi] - \phi ) \Big] (t,y) \D y \to 0 \ \mbox{in} \ L^\infty(0,T).
\]
Consequently, using Arzel\' a--Ascoli theorem and the bounds on $D_y \vc{X}_\ep$ established in 
\eqref{v6b}, we obtain 
\[
\left[ t \mapsto \int_{\Omega} \Big[ [D_y \vc{X}_\ep ] \cdot \vc{v}_\ep \cdot [D_y \vc{X}_\ep ] \phi \Big] (t,y) \D y \right] \to \left[ 
t \mapsto \int_{\Omega} [\vc{v} \cdot  \phi]  (t,y) \D y \right]\ \mbox{in}\ C[0,T] 
\]
for any $\phi \in \DC(\Omega; R^3)$, ${\rm div} \phi = 0$. Finally, using again
\eqref{v6b}, we may infer that
 	 \begin{equation} \label{A9}
 	 \left[ t \mapsto \int_{\Omega} [\vc{v}_\ep  \phi ] (t,y) \D y \right] \to \left[ 
 	 t \mapsto \int_{\Omega} [\vc{v} \cdot  \phi]  (t,y) \D y \right]\ \mbox{in}\ C[0,T] 
 	 \end{equation}
for any $\phi \in \DC(\Omega)$, ${\rm div} \phi = 0$. In view of  
the standard Aubin--Lions argument, the convergences \eqref{A5} and \eqref{A9} 
yield \eqref{A6}.  

Finally, observe that the same argument yields also the convergence of the initial 
	values
\[
\vu_{0, \ep} = \vc{v}_\ep (0, \cdot) \to \vu_0 \ \mbox{weakly in}\ L^2(\Omega; R^3).
\]

\subsection{Time regularization and homogenization}

Following the idea of \cite{BFN4} we regularize \eqref{N14} by means of a convolution with a time dependent families of regularizing kernels $\{ \theta_\delta \}_{\delta > 0}$. This means we consider the functions 
\[
\theta_{\delta}(\tau - t) \phi(x) ,\ \phi \in W^{1,2}_0 (\Ome; R^3),\ 
{\rm div}_y \phi = 0,\ \tau \in (\delta, T -\delta), 
\]
as test functions in \eqref{N14}. We denote 
\[
[f]_\delta (\tau, y) = \theta_\delta * f (\tau, y) = \int_0^T \theta_\delta (\tau - t) f(t,y) \ \dt.
\] 

The weak formulation \eqref{N14} gives rise to a family of \emph{stationary} Stokes problems for any fixed $\delta$ and $\tau \in (0,T)$:
\begin{align} 
	\int_{\Omega} \Grad [\vc{v}_\ep ]_\delta (\tau, \cdot) : \Grad \phi \ \D y &= 
- \int_{\Omega} \partial_t \Big[ [D_y \vc{X}_\ep ] (t,y) \cdot  \vc{v}_\ep (t,y)  \cdot 
[D_y \vc{X}_\ep] (t,y) \Big]_\delta (\tau, \cdot) \cdot \bfphi \D y \br 
&+ \int_{\Omega} [\widetilde{\vu}_\ep  \otimes \widetilde{\vu}_\ep ]_\delta  : 
\nabla_y \phi \ \D y + \left< \vc{g}_{\ep, \delta}(\tau, \cdot); \phi \right>
\label{A10}
\end{align}
for any 
\[
\phi \in W^{1,2}_0(\Omega_\ep; R^3),\ {\rm div}_y \phi = 0.
\]
Moreover, in view of the bounds established in \eqref{A2}, 
\[
\vc{g}_{\ep, \delta}(\tau, \cdot) \to 0 \ \mbox{in}\ 
W^{-1,2}(\Omega; R^3) \ \mbox{as}\ \ep \to 0 
\ \mbox{for any fixed}\ \delta > 0, \tau \in (0,T).
\]

Seeing that the regularized velocity $[\vc{v}_\ep]_\delta$ still satisfies the boundary conditions \eqref{N8} (with regularized rigid boundary velocities), we may apply Proposition \ref{Pp1} for any fixed $\tau > 0$. Keeping $\delta > 0$ and $\tau \in (0,T)$ fixed, we let $\ep \to 0$ in \eqref{A10} obtaining
\begin{align} 
	\int_{\Omega} \Grad [\vc{v} ]_\delta (\tau, \cdot) : \Grad \phi \ \D y &= 
	- \int_{\Omega} \partial_t [ \vc{v} ]_\delta (\tau, \cdot) \cdot \phi\ \D y \br &+ \int_{\Omega} [\vc{v}  \otimes \vc{v} ]_\delta (\tau, \cdot)  : 
	\nabla_y \phi \ \D y + \int_{\Omega} 6 \pi \mathfrak{R} [\vc{v}]_\delta (\tau, \cdot) \cdot \phi \ \D y \dt
	\label{A11}
\end{align}
for any 
\[
\phi \in W^{1,2}_0(\Omega; R^3),\ {\rm div}_y \phi = 0.
\]
Here we have used the strong convergence of $\vc{v}_\ep$ (hence also $\widetilde{\vu}_\ep$) established in \eqref{A6} to pass to the limit in the quadratic term $[\widetilde{\vu}_\ep  \otimes \widetilde{\vu}_\ep ]_\delta$.

Finally, we let $\delta \to 0$ in \eqref{A11}, and use the identity of the limits $\vu = \vc{v}$ established in Section \ref{scv} 
to conclude the proof of Theorem \ref{TP1}. 
\section{A priori estimate of the velocity of small rigid bodies}\label{sec7}

We prove Proposition \ref{Impr:contr}, that implies a control of the velocity of a the rigid bodies using an idea from \cite{BraNec2}.

\begin{proof}[Proof of Proposition \ref{Impr:contr}]

Suppose that there exists $ \bar{\ep} > 0 $ such that for any $ 0 < \ep < \bar{\ep}$ 
\begin{equation}
\label{hyp:21}
\sup_{t \in [0,T ]} |\vc{h}_{i, \ep}(t) - \vc{h}_{j,\ep}(t) | >  \ep,\ i\ne j,\ 	
{\rm dist}[\vc{h}_{n,\ep}(0); \partial \Omega] > \frac{1}{2}\ep.
\end{equation} 
Then for $ 0 <  \ep < \bar{\ep} $, we start by recalling a cut--off function introduced in Section \ref{v}. 
\[
\chi \in C^\infty_c [0, 1),\ \chi(z) = \left\{ \begin{array}{l} 1 
\ \mbox{if} \ 0 \leq z < \frac{1}{2}, \\ 
\in [0,1] \ \mbox{if} \ \frac{1}{2} \leq z \leq \frac{3}{4}, \\ 
0 \ \mbox{if}\ z > \frac{3}{4}.	 
	\end{array} \right.
\]
Moreover for $ z \in R^3 $, we define the velocity field
\begin{equation} \label{cut:off:2}
	\Psi_{r_\ep} [z](x) = \Curl \left[ \chi \left(2 \frac{|x |}{ r_{\ep}} \right) 	
	\mathbb{T}( x) \cdot  z  \right],  
\end{equation}	
where $ \mathbb{T} $ is defined in \eqref{def:T}.

Let $ \gamma \in C^{\infty}_c[0,T) $ and testing \eqref{P23} with 
\begin{equation}
\psi_{n,\ep} = \frac{1}{m_{n,\ep}} \Psi_{r_\ep} \left[\int_t^T \gamma(\tau) \ \D \tau \right](x-\vc{h}_{n,\ep}(t)),
\end{equation} 
we deduce 
\begin{align*}
\int_0^T & (\vc{h}_{n,\ep}'(t)-\vc{h}_{n,\ep}'(0))\cdot \gamma(t) \ \D t = \int_0^T \int_{\Omega_{t,\ep}} \vc{u}_{\ep} \cdot \partial_t \psi_{n,\ep} + [\vue \otimes \vue]:  \nabla_x \psi_{n,\ep}  \ \D x \D t \\
& - \int_0^T \int_{\Omega_{\ep,t}} \nabla_x \vue : \nabla_x \psi_{n,\ep}   \ \D x \D t + \int_{\Omega_{0,\ep}}  \vc{u}_{\ep,0}  \cdot \psi_{n,\ep}(0,.) \ \D x . 
\end{align*}
We compute
\begin{align*}
\partial_t \psi_{n,\ep}(t,x) = & \, -\frac{1}{m_{n,\ep}}\Psi_{r_{\ep}} \left[ \gamma(t)\right](x-\vc{h}_{n,\ep}(t)) \\
& \, -\vc{h}_{n,\ep}'(t) \cdot \nabla_x \left(\frac{1}{m_{n,\ep}} \Psi_{r_{\ep}} \left[ \int_t^T \gamma(\tau ) \ \D \tau \right] \right)(x-\vc{h}_{n,\ep}(t)).
\end{align*}
We deduce, for a constant $ C $ independent of $ \ep $ and $ n $
\begin{equation}
\label{M1}
\|\partial_t \psi_{n,\ep}(t,x) \|_{L^1(0,T;L^2(R^3)))} \leq C \frac{r_{\ep}^{3/2}}{m_{n,\ep}}\|\gamma \|_{L^1(0,T)} + C \frac{r_{\ep}^{1/2}}{m_{n,\ep}}\| \vc{h}_{n,\ep}' \|_{L^1(0,T)}  \|\gamma \|_{L^1(0,T)} 
\end{equation}
and similarly
\begin{align}
\label{M2}
\| \nabla_x \psi_{n,\ep}(t,x) \|_{L^{\infty}(0,T;L^2(R^3)))} \leq C \frac{r_{\ep}^{1/2}}{m_{n,\ep}} \|\gamma \|_{L^1(0,T)}. 
\end{align}
Korn inequality together with the Sobolev inequality imply that
\begin{equation}
\label{M3}
\| \vue \|_{L^2(0,T; L^6(\Omega))} \leq C \| D \vue \|_{L^2(0,T; L^2(\Omega))} \leq C,
\end{equation}  
by the energy estimates. The bounds \eqref{M1}-\eqref{M2}-\eqref{M3} imply 
\begin{align*}
\int_0^T (\vc{h}_{n,\ep}'(t)- & \vc{h}_{n,\ep}'(0))\cdot \gamma(t) \ \D t  \leq  C \Bigg( \frac{r_{\ep}^{3/2}}{m_{n,\ep}} \|\vue \|_{L^2(0,T;L^2(\Omega))} + \frac{r_{\ep}^{1/2}}{m_{n,\ep}} \|\vue \|_{L^2(0,T;L^6(\Omega))}^2 \\ 
& \, + \sqrt{T}\frac{r_{\ep}^{1/2}}{m_{n,\ep}} \|\nabla_x \vue \|_{L^2(0,T;L^2(\Omega))}  + \frac{r_{\ep}^{3/2}}{m_{n,\ep}} \| \vc{u}_{\ep,0}\|_{L^2(\Omega)} +  \frac{r_{\ep}^{1/2}}{m_{n,\ep}} \|\vue \|_{L^2(0,T;L^2(\Omega))}T | \vc{h}_{n,\ep}'(0)| \Bigg)\|\gamma\|_{L^1(0,T)} \\
& \, +C  \frac{r_{\ep}^{1/2}}{m_{n,\ep}} \|\vue \|_{L^2(0,T;L^2(\Omega))}T \| \vc{h}_{n,\ep}'- \vc{h}_{n,\ep}'(0)\|_{L^{\infty}(0,T)}  \|\gamma \|_{L^1(0,T)}.
\end{align*}
Dividing by $ \| \gamma \|_{L^1(0,T )} $ and taking the sup over all $  \gamma  \in C^{\infty}_c[0,T) $, we deduce
\begin{align*}
\| \vc{h}_{n,\ep}' - \vc{h}_{n,\ep}'(0)  \|_{L^{\infty}(0,T)} \leq &   C \frac{r_{\ep}^{1/2}}{m_{n,\ep}}  
 +C  \frac{r_{\ep}^{1/2}}{m_{n,\ep}} \|\vue \|_{L^2(0,T;L^2(\Omega))}T \| \vc{h}_{n,\ep}' - \vc{h}_{n,\ep}'(0) \|_{L^{\infty}(0,T)}.
\end{align*}
From hypothesis \eqref{H1}, we can absorb the second term of the right hand side and deduce the desired result. 

Similarly for the angular velocity, we consider, for $ z \in R^3 $, a auxiliary function
\begin{equation}
	\Phi_{r_\ep} [z](x) = \Curl \left[ - \chi \left(2 \frac{|x |}{ r_{\ep}} \right) 	
	 \frac{|x|^2}{2} z  \right],  
\end{equation}	
\begin{equation}
\phi_{n,\ep} = \frac{1}{r_{\ep} m_{n,\ep}} \Phi_{r_\ep} \left[\int_t^T  \gamma(\tau) \ \D \tau \right](x-\vc{h}_{n,\ep}(t)).
\end{equation}  
In the spirit of \eqref{M1} and \eqref{M2}, we have 
\begin{equation}
\label{M11}
\|\partial_t \phi_{n,\ep}(t,x) \|_{L^1(0,T;L^2(R^3)))} \leq C \frac{r_{\ep}^{3/2}}{m_{n,\ep}}\|\gamma \|_{L^1(0,T)} + C \frac{r_{\ep}^{1/2}}{ m_{n,\ep}}\| \vc{h}_{n,\ep}' \|_{L^1(0,T)}  \|\gamma \|_{L^1(0,T)} 
\end{equation}
and similarly
\begin{align}
\label{M22}
\| \nabla_x \phi_{n,\ep}(t,x) \|_{L^{\infty}(0,T;L^2(R^3)))} \leq C \frac{r_{\ep}^{1/2}}{ m_{n,\ep}} \|\gamma \|_{L^1(0,T)}. 
\end{align}
We deduce that 
\begin{align*}
\frac{2}{5}r_{\ep}&\int_0^T (\bfomega_{n,\ep}'(t)-  \bfomega_{n,\ep}'(0))\cdot \gamma(t) \ \D t \\ \leq & \,     C\|\gamma\|_{L^1(0,T)} \left(\frac{r_{\ep}^{1/2}}{ m_{n,\ep}}  
 +C  \frac{r_{\ep}^{1/2}}{ m_{n,\ep}} \|\vue \|_{L^2(0,T;L^2(\Omega))}T \| \vc{h}_{n,\ep}'(t)\|_{L^{\infty}(0,T)}\right) \\
\leq & \,  C\|\gamma\|_{L^1(0,T)} \frac{r_{\ep}^{1/2}}{ m_{n,\ep}}. 
\end{align*}
Dividing by $ \|\gamma\|_{L^1(0,T)} $ and taking the sup over all possible $ \gamma $, we deduce the result. 

We conclude the proof by showing that hypothesis \eqref{hyp:21} holds by contradiction. 
Suppose by contradiction that  \eqref{hyp:21} does not hold. There exists a sequence $ \ep_k \to 0 $ such that  \eqref{hyp:21} is not satisfied. By \eqref{H1} there exists $ \bar{k} $ such that for any $ k > \bar{k} $ we have
\begin{equation}
\label{hyp:k:bar}
|\vc{h}_{n,\ep_k}'(0)| < \frac{\ep_k}{8T} \quad \text{ and } \quad   \frac{r_{\ep_k}^{1/2}}{\ep_k m_{n,\ep_k}} < \frac{1}{4CT}.
\end{equation}
Let now $ t_k $ the infimum in $ [0,T] $ such that \eqref{hyp:21} with $ \ep = \ep_k $ does not hold. By continuity of $ \vc{h}_{n,\ep_k } $ and the fact that $ N(\ep_k) < \infty $ the infimum is a minimum. In $ [0,t_k) $ the bound \eqref{bound:velocity} holds. We deduce that
\begin{align}
|\vc{h}_{i,\ep_k}(t_k) - \vc{h}_{j,\ep_k}(t_k)| \geq & \, |\vc{h}_{i,\ep_k}(0) - \vc{h}_{j,\ep_k}(0)|   - \left| \int_0^{t_k} \vc{h}_{i,\ep_k}'(\tau) \ \D  \tau \right| - \left| \int_0^{t_k} \vc{h}_{j,\ep_k}'(\tau) \ \D  \tau \right| \nonumber \\
\geq & \,  2 \ep_k - \frac{\ep_k}{8}- \frac{\ep_k}{8}- \frac{\ep_k}{8}- \frac{\ep_k}{8} =  \frac{3}{2}\ep_k. \label{contr:1}
\end{align}
Similarly we have
\begin{equation}
\label{contr:2}
{\rm dist}[\vc{h}_{n,\ep_k}(t); \partial \Omega] \geq {\rm dist}[\vc{h}_{n,\ep_k}(0); \partial \Omega]- \left| \int_0^{t_k} \vc{h}_{n,\ep_k}'(\tau) \ \D  \tau \right| \geq \frac{3}{4}\ep_k.
\end{equation}
Then \eqref{contr:1} and \eqref{contr:2}, contradict that \eqref{hyp:21} does not holds. The proof is then completed.

\end{proof}

\section{Concluding remarks}
\label{CC}

The results were stated in the space dimension $d=3$. Extension to $d=2$ can be done 
with the appropriate scaling of the radius $r_\ep$, 
specifically, 
\[
- \ep^2 \log(r_\ep) \to C \in (0, \infty). 
\]
The homogenization limit for the stationary Stokes problem stated in Proposition 
\ref{Pp1} can be extended to the periodic distribution of obstacles by the method of Allaire 
\cite{Allai4}. Rather general distribution of obstacles and their radii was considered 
by Marchenko and Khruslov \cite[Chapter 4, Theorem 4.7]{MarKhr}, cf. also Namlyeyeva et al. \cite{NamNecSkr}. We leave the details to the interested reader.


\def\cprime{$'$} \def\ocirc#1{\ifmmode\setbox0=\hbox{$#1$}\dimen0=\ht0
	\advance\dimen0 by1pt\rlap{\hbox to\wd0{\hss\raise\dimen0
			\hbox{\hskip.2em$\scriptscriptstyle\circ$}\hss}}#1\else {\accent"17 #1}\fi}

\end{document}